\documentclass{amsart}
\usepackage{hyperref}

\usepackage{enumerate}
\usepackage{enumitem}
\setitemize{leftmargin=1cm}
\setenumerate{leftmargin=7mm}
\newcommand{\arxiv}[1]{\href{http://arxiv.org/abs/#1}{arXiv:#1}}
\newcommand*{\mailto}[1]{\href{mailto:#1}{\nolinkurl{#1}}}

\newtheorem{theorem}{Theorem}[section]

\newtheorem{lemma}[theorem]{Lemma}
\newtheorem{corollary}[theorem]{Corollary}

\newtheorem*{proposition*}{Proposition}

\theoremstyle{definition}
\newtheorem{remark}[theorem]{Remark}

\newcommand{\R}{\mathbb{R}}

\newcommand{\N}{\mathbb{N}}
\newcommand{\C}{\mathbb{C}}

\newcommand{\F}{\mathcal{F}}
\newcommand{\RR}{\mathcal{R}}

\newcommand{\be}{\begin{equation}}
\newcommand{\ee}{\end{equation}}
\newcommand{\beq}{\begin{equation}}
\newcommand{\eeq}{\end{equation}}
\newcommand{\bea}{\begin{eqnarray}}
\newcommand{\eea}{\end{eqnarray}}

\newcommand{\ol}{\overline}
\newcommand{\pa}{\partial}
\newcommand{\ti}{\tilde}

\newcommand{\id}{\mathbb{I}}
\newcommand{\I}{\mathrm{i}}
\newcommand{\E}{\mathrm{e}}

\newcommand{\re}{\mathop{\mathrm{Re}}}

\newcommand{\noprint}[1]{}

\makeatletter
\newcommand{\dlmf}[1]{%
\cite[%
  \def\nextitem{\def\nextitem{, }}%
  \@for \el:=#1\do{\nextitem\href{http://dlmf.nist.gov/\el}{(\el)}}%
]{dlmf}%
}
\makeatother

\newcommand{\la}{\lambda}

\numberwithin{equation}{section}

\begin{document}
\title[Soliton asymptotics for KdV shock waves]{Soliton asymptotics for KdV shock waves\\ via classical inverse scattering}

\author[I. Egorova]{Iryna Egorova}
\address{B. Verkin Institute for Low Temperature Physics and Engineering\\ 47, Nauky ave\\ 61103 Kharkiv\\ Ukraine}
\email{\href{mailto:iraegorova@gmail.com}{iraegorova@gmail.com}}

\author[J. Michor]{Johanna Michor}
\address{Faculty of Mathematics\\ University of Vienna\\
Oskar-Morgenstern-Platz 1\\ 1090 Wien\\ Austria}
\email{\href{mailto:Johanna.Michor@univie.ac.at}{Johanna.Michor@univie.ac.at}}
\urladdr{\href{http://www.mat.univie.ac.at/~jmichor/}{http://www.mat.univie.ac.at/\string~jmichor/}}

\author[G. Teschl]{Gerald Teschl}
\address{Faculty of Mathematics\\ University of Vienna\\
Oskar-Morgenstern-Platz 1\\ 1090 Wien\\ Austria\\ and Erwin Schr\"odinger International
Institute for Mathematics and Physics\\ Boltzmanngasse 9\\ 1090 Wien\\ Austria}
\email{\href{mailto:Gerald.Teschl@univie.ac.at}{Gerald.Teschl@univie.ac.at}}
\urladdr{\href{http://www.mat.univie.ac.at/~gerald/}{http://www.mat.univie.ac.at/\string~gerald/}}

\thanks{Research supported by the Austrian Science Fund (FWF) under Grant No.\ P31651.}

\begin{abstract}
We show how the inverse scattering transform can be used as a convenient tool to derive the
long-time asymptotics of shock waves for the Korteweg--de Vries (KdV) equation in the soliton
region. In particular, we improve the results previously obtained via the nonlinear
steepest decent approach both with respect to the decay of the initial conditions as well as
the region where they are valid.
\end{abstract}

\maketitle

\section{Introduction and main results}

Presently the most common method for studying the long-time asymptotic behavior of solutions of completely integrable nonlinear wave equations is 
the nonlinear steepest descent (NSD) method introduced by Deift and Zhou \cite{DZ} extending ideas of Manakov \cite{ma} and Its \cite{its1}.
In particular, this method has superseded the inverse scattering transform (IST) originally used in this context. The purpose of the present note
is to address two issues arising in the context of the NSD method and show how they can be handled, as we feel, more effectively using the IST, at least
in the soliton region. Hence it should be emphasized that the main contribution of this note is not merely the improvement of previous results but the
techniques which lead to these improvements. Clearly these techniques are not limited to our chosen example but apply to other integrable wave equations
solvable via the IST equally well.

The first item is the fact that a common assumption in the context of the NSD method is exponential decay of the initial condition such that the scattering data
allow for an analytic continuation to a neighborhood of the spectrum. Of course this assumption can be weakened by using an analytic approximation of
the scattering data and this was already demonstrated in \cite{DZ}, where Schwartz-type initial conditions were considered. However, it should be pointed out that this technique
is not limited to Schwartz data and can in principle also handle weaker conditions of decay (see \cite{GT} for the case of the KdV equation with decaying initial data). All of this can be considered
well-understood. Nevertheless, working out all steps in full detail requires considerable effort and hence is frequently skipped.

The second item is the fact that the NSD method produces the asymptotics in sectors. While this is of course natural, since the asymptotics are different for different regions,
it still has the drawback that the different sectors usually do not overlap. In particular, we are not aware of an analysis which produces the asymptotics of the KdV
equation, in the simplest case of decaying initial condition, in overlapping domains which cover the entire $x/t$-plane. In this context our second contribution is
that we can establish soliton asymptotics in a regime which is larger than what could previously be covered using the NSD method.

In this vein, the aim of our note is to show the benefits of the classical IST in soliton domains for the study of the KdV equation
\be\label{kdv} 
q_t(x,t) - 6 q(x,t)q_x(x,t)+q_{xxx}(x,t)=0,
\ee
with step-like initial data $q(x)=q(x,0)$ satisfying the condition
\be\label{ini1}
q(x)\to\begin{cases} \begin{array}{ll} 0, &\mbox{as}\ x\to +\infty,\\
-c^2, &\mbox{as}\ x\to -\infty,\end{array}\end{cases}
\ee
 where $c>0$. The long-time asymptotics of solutions of \eqref{kdv}, \eqref{ini1}
 were derived in \cite{EGKT} using the vector Riemann--Hilbert approach,  under the assumption of exponential decay of the initial data to the background constants in \eqref{ini1}. We want to weaken this condition to polynomial decay
 and, at the same time, cover a larger region. Specifically, we will derive the asymptotics for
\[
x \geq 4c^2t + \frac{m_0-\frac 3 2 -\varepsilon}{2c}\,\log t, \quad \varepsilon >0, \quad m_0 \geq 3,
\]
whereas the asymptotics obtained via NSD in \cite{EGKT} only hold for 
\[x\geq 4c^2 t + \varepsilon t, \quad \varepsilon >0.
\]
Our second aim was to find minimal restrictions on the steplike initial data  for which the 
associated  Riemann--Hilbert problem (RHP) is well-posed. When stating the RHP for a nonlinear equation, 
one assumes that its solution for each $t$ has a behavior that guarantees the existence of Jost solutions. 
In our case, it means that the solution $q(x,t)$ to the initial-value problem \eqref{kdv}--\eqref{ini1} has to have finite first moments of perturbation. If this is not established 
beforehand, then the respective RHP statement is conditional, and one has to prove that the RHP has a solution. 
These minimal restrictions on the initial data to achieve existence of a unique classical solution are the content of Theorem~\ref{thoern}. 

Our main result is the following. 
Denote by $C^n(\R)$ the set of functions $x \in \R \mapsto q(x) \in \R$ which have $n$ continuous derivatives with respect to $x$
and let $C^{n,k}(\R^2)$ be the set of functions $(x,t) \in \R^2 \mapsto q(x,t) \in \R$ which have $n$ continuous derivatives with respect to $x$ 
and $k$ continuous derivatives with respect to $t$. Denote the discrete spectrum of the associated scattering problem by 
$-\kappa_N^2<\dots<-\kappa_1^2$ and the inverse of the norm of the right eigenfunctions by 
$\gamma_j$, $j=1,\dots,N$ (compare Sec.\ \ref{S:scattering}).

\begin{theorem} \label{maintheor1} Let $m_0\geq 3$, $n_0\geq 2$ be arbitrary natural numbers and let $\varepsilon>0$ be arbitrary small. Assume that $q(x)\in C^{n_0}(\R)$ satisfies
\be\label{inithor}\int_0^{+\infty} (1 +|x|^{m_0})\left(\left|\frac{d^n}{d x^n}q(x)\right| +\left|\frac{d^n}{d x^n}(q(-x)+c^2)\right|\right)dx<\infty,\quad  0\leq n\leq n_0.\ee Let $q(x,t)$ be the solution of \eqref{kdv}, \eqref{inithor}. 
Then the asymptotics of $q(x,t)$ in the  region 
\[x \geq 4c^2t + \frac{m_0-\frac 3 2 - \varepsilon}{2c}\,\log t\]  
is the following as $t\to\infty$:
\be \label{asmain}
q(x,t) 
=-\sum_{j=1}^N\frac{2\kappa_j^2}{\cosh^2\big(\kappa_j x - 4\kappa_j^3 t -\frac{1}{2}\log\frac{\gamma_j^2}{2\kappa_j}-\sum_{i=j+1}^N\log\frac{\kappa_j - \kappa_i}{\kappa_i + \kappa_j}\big)} + O\bigg(\frac{1}{t^{m_0-\frac{3}{2}-\varepsilon}}\bigg).
\ee
\end{theorem}
For convenience, we recall the precise result from \cite{EGKT}.

\begin{proposition*}[\cite{EGKT}]
Assume that the initial data $q(x)$ of the problem \eqref{kdv}--\eqref{ini1} belongs to $C^7(\R)$ and satisfies
\be\label{ini2}\int_0^{+\infty}\E^{C_0 x}(|q(x)| + |q(-x) + c^2|) dx <\infty, \quad q^{(i)}x^4\in L_1(\R),
\ee
where $i=1,\dots,7$ and $C_0>c>0$. Let $q(x,t)$ be the solution of \eqref{kdv}, \eqref{ini2}. 
Assume that $\delta>0$ is sufficiently small such that the intervals $[4\kappa_j^2 -\delta, 4\kappa_j^2 +\delta]$ are disjoint and lie inside $(4c^2, \infty)$. Then the asymptotics in the soliton region $\frac{x}{t} - 4c^2\geq \varepsilon$ for some small $\varepsilon>0$ are as follows:

If $x\to\infty$, $t\to\infty$ and $\left|\frac{x}{t} - 4\kappa_j^2\right|<\delta$ for some $j$, 
\be\label{as1}
q(x,t)=\frac{-4\kappa_j\gamma_j^2(x,t)}{(1+(2\kappa_j)^{-1}\gamma_j^2(x,t))^2}+O(t^{-l}),\ee
for any $l\in\mathbb N$, where
\be\label{gammaxt}
\gamma_j^2(x,t)=\gamma_j^2\E^{-2\kappa_j x+8\kappa_j^3 t}\prod_{j=j+1}^N\left(\frac{\kappa_i -\kappa_j}{\kappa_i + \kappa_j}\right)^2.
\ee

If  $\left|\frac{x}{t} - 4\kappa_j^2\right|>\delta$ for all $j$, then
$q(x,t)=O(t^{-l})$ for all $l \in\mathbb N$.
\end{proposition*}

\begin{remark} Let us check that \eqref{asmain} and \eqref{as1} indeed represent the same leading term 
of the asymptotic expansion. For sufficiently large $t$ the summands in \eqref{asmain} do not interfere up to an exponentially small term. They have nonzero profiles in vicinities of the rays $\frac{x}{t}=4\kappa_j^2$ defined by the phase shifts
\be\label{deltaj}\Delta_j=-\frac{1}{2}\log\frac{\gamma_j^2}{2\kappa_j}-\sum_{i=j+1}^N\log\frac{\kappa_j - \kappa_i}{\kappa_i + \kappa_j}.\ee
On the other hand, using \eqref{gammaxt} and \eqref{deltaj} we get
\begin{align*}
\frac{-4\kappa_j\gamma_j^2(x,t)}{(1+(2\kappa_j)^{-1}\gamma_j^2(x,t))^2}&=-\frac{8\kappa_j^2}{\left(\frac{\sqrt{2\kappa_j}}{\gamma_j(x,t)} +\frac{\gamma_j(x,t)}{\sqrt{2\kappa_j}}\right)^2} \\
&= -\frac{8\kappa_j^2}{\left(\E^{\kappa_j x- 4\kappa_j^3 t + \Delta_j} +\E^{-\kappa_j x+4\kappa_j^3 t - \Delta_j}\right)^2},
\end{align*}
 which implies \eqref{asmain}.
\end{remark}
Existence of the unique classical solution of the initial-value problem \eqref{kdv}, \eqref{inithor} considered in Theorem \ref{maintheor1} is provided by the results of Grudsky and Rybkin \cite{GR} which  are applicable for the wider class of initial data essentially bounded from below on the left half axis and decaying with rate $x^{5/2} q(x)\in L_1(\R_+)$ on the right half axis. Recently these results were generalized by 
Laurens \cite{Lau1, Lau2} for initial data, which are the sum of a known steplike KdV solution and an arbitrary function
from $H^{-1}(\R)$. However, all these results do not allow us to control the behavior of $q(x,t)$ as $x\to \pm\infty$ for any fixed $t$, and they do not ensure the requirements on smoothness of the scattering data applied in the  asymptotic analysis. This can be achieved in the framework of the classical IST generalized to the case of steplike solutions.  For steplike KdV solutions the following result holds:

\begin{proposition*}[\cite{EGT, ET3}] 
Let $p_\pm(x)$ be two different algebro-geometric quasi-periodic finite gap potentials and let $p_\pm(x,t)$ be two finite gap KdV solutions associated with the initial data $p_\pm(x,0)=p_\pm(x)$. Assume that $q(x,0)\in C^{n_0}(\R)$ and
\be\label{iniper}\pm \int_0^{\pm\infty}\left|\frac{d^n}{dx^n}(q(x,0) - p_\pm(x))\right|(1 + |x|^{m_0})dx <\infty,\quad 0\leq n\leq n_0,
\ee 
where 
\[ 
m_0\geq 8, \quad \mbox{and}\quad n_0\geq m_0+5,
\]
are some fixed natural numbers. Then there exists a unique classical solution $q(x,t)\in C^{n_0 - m_0 -2,1}(\R^2)$ of the initial-value problem \eqref{kdv}, \eqref{iniper} satisfying
\[ \label{deriv}
\pm\int_0^{\pm\infty}\left|\frac{\partial^n}{\partial x^n}(q(x,t) - p_\pm(x,t))\right| \Big(1 + |x|^{[\frac{m_0}{2}]-2}\Big)dx <\infty,
\quad  0\leq n\leq n_0 - m_0 -2,
\]
and
\[ \label{dert}
\pm\int_0^{\pm\infty}\left|\frac{\partial}{\partial t}(q(x,t) - p_\pm(x,t))\right| \Big(1 + |x|^{[\frac{m_0}{2}]-2}\Big)dx <\infty,
\]
for all $t\in\R$.
\end{proposition*}

Hence the largest class of initial data for which unique solvability is established, corresponds to initial datum with $13$ continuous derivatives and $8$ finite moments of perturbation. In our case of constant background solutions
$p_+(x)=p_+(x,t)=0$ and $p_-(x)=p_-(x,t)=-c^2$, we can extend this class 
to initial datum with $7$ continuous derivatives and $4$ finite moments.

\begin{theorem} \label{thoern}  
Given a pair  $(m_0, n_0)\in \N\times\N$ such that
\be\label{estbig}
m_0\geq 4, \quad \mbox{and}\quad n_0\geq m_0+3,\ee assume that  $q(x)\in C^{n_0}(\R)$ and satisfies \eqref{inithor}.
Then there exists a unique classical solution to the initial-value problem \eqref{kdv}, \eqref{inithor}, \eqref{estbig} such that for all $t\in\R$
\be\label{inithor1}\int_0^{+\infty} \Big(1 +|x|^{[\frac{m_0}{2}]-1}\Big)\left(\left|\frac{\partial^n}{\partial x^n}q(x,t)\right| +\left|\frac{\partial^n}{\partial x^n}(q(-x,t)+c^2)\right|\right)dx<\infty,\ee
for all $ 0\leq n\leq n_0-m_0,$ and
\be\label{inithor2}\int_\R \Big(1 +|x|^{[\frac{m_0}{2}]-1}\Big)\left|\frac{\partial}{\partial t}q(x,t)\right| dx<\infty.
\ee
\end{theorem}
Thus condition \eqref{estbig} ensures the well-posedness of the  Riemann--Hilbert problem which forms the base of the 
long-time asymptotic analysis of steplike solutions via the NSD approach. It also secures the smoothness and decay 
properties of the scattering data required for the asymptotic analysis used in Theorem~\ref{maintheor1}.

\section{Scattering theory for the Schr\"{o}dinger operator}  \label{S:scattering}

In this section we briefly describe the characteristic (necessary and sufficient) properties of the scattering data for the Schr\"{o}dinger operator
\[ (Lf)(x):=-\frac{d^2}{dx^2}f(x) + q(x)f(x)\] with steplike potential \eqref{ini1} satisfying \eqref{inithor}, assuming that $m_0\geq 1$ and $n_0\geq 0$. Let $\la$ be the spectral parameter of the problem 
\be\label{sch} -\frac{d^2}{dx^2}f(x) + q(x)f(x)=\la f(x).\ee
We introduce additional spectral parameters $k=\sqrt\la$ and $k_1=\sqrt{\la + c^2}$ which map the domains $\C\setminus \R_+$ and $\mathcal D:=\C\setminus [-c^2, \infty)$ conformally onto $\C^+$. The boundary of the domain $\mathcal D$ is treated as consisting of two sides of the cut along the interval $[-c^2, \infty)$, with distinguished points $\la^u=\la+\I 0$ and $\la^l=\la- \I 0$ on this boundary. We assume that $\la\in\overline {\mathcal D}$. This gives a one-to-one correspondence between the parameters $k$, $k_1$ and $\la$. 

The spectrum of the operator $L$ consists of a continuous spectrum of multiplicity two on $\R_+$, a continuous spectrum of multiplicity one on $[-c^2, 0]$ and a finite discrete spectrum $\{\la_1, \dots, \la_N\}$, where $\la_N<\dots<\la_1<-c^2.$

According to \cite{M} and \cite{DT}, equation \eqref{sch} has two Jost solutions
\[\aligned \phi(\la,x)&=\E^{\I k x}\left(1 +\int_0^\infty B(x,y)\E^{2\I k y} dy\right), \quad k\in\ol{\C^+},  \\ \phi_1(\la,x)&=\E^{-\I k_1 x}\left(1 +\int^0_{-\infty} B_1(x,y)\E^{-2\I k_1 y} dy\right),\quad k_1\in\ol{\C^+}, 
\endaligned
\]
where $B(x,y)$ and $B_1(x,y)$ are real-valued functions (the kernels of the transformation operators) satisfying
the properties
\be\label{form1}
B(x,0)=\int_x^{+\infty} q(y) dy,\quad B_1(x,0)=\int_{-\infty}^x (q(y)+c^2)dy.\ee
One of the main characteristics of the spectral problem \eqref{sch}, \eqref{inithor} is the Wronskian
$W(\la):=W(\phi(\la,\cdot), \phi_1(\la,\cdot))$ of the Jost solutions, where 
$W(f,g)=f(x)g^\prime(x) - g(x)f^\prime(x)$. It is well-known that $W(\la)$ has simple zeros at 
the points of the discrete spectrum $\la_j$ and possibly at the end of the continuous spectrum, but no other zeros. 
If $W(-c^2)=0$ we call the point $-c^2$ a resonant point. Denote by
\[\gamma_j^2:=\frac{1}{\int_{\R}\phi^2(\la_j,x) dx},\quad \gamma_{j,1}^2:=\frac{1}{\int_{\R}\phi_1^2(\la_j,x) dx}\]
the right, left normalizing constants, respectively.

Consider the standard  scattering relations
\[\aligned T(\la)\phi_1(\la, x)&=\ol{\phi(\la,x)} + R(\la)\phi(\la,x),\quad k\in\R,\\
T_1(\la)\phi(\la, x)&=\ol{\phi_1(\la,x)} + R_1(\la)\phi_1(\la,x),\quad k_1\in\R,\endaligned\]
where the transmission and reflection coefficients constituting the entries of the scattering matrix are 
defined as usual
\be\aligned T(\la)&:=\frac{W(\ol{\phi(\la)}, \phi(\la))}{W(\la)},\quad R(\la):=\frac{W(\ol{\phi(\la)}, \phi_1(\la))}{W(\la)}, \quad \quad k\in\R,\\
T_1(\la)&:=-\frac{W(\ol{\phi_1(\la)}, \phi_1(\la))}{W(\la)},\quad R_1(\la):=-\frac{W(\ol{\phi_1(\la)}, \phi(\la))}{W(\la)}, \quad k_1\in\R.
\endaligned\ee

\begin{lemma}[\cite{BET, CK, EGLT}] \label{1-3}
Let $q(x)$ satisfy \eqref{inithor} with $m_0\geq 1$ and $n_0\geq 0$.
The entries of the scattering matrix satisfy
\vskip 3mm
\begin{itemize}
\item[{\bf I.(a)}] $T(\lambda+\I 0) =\overline{T(\lambda-\I 0)}$ and
$R(\lambda+\I 0) =\overline{R(\lambda-\I 0)}$ for $k(\la)\in\R$; \newline
$T_1(\lambda+\I 0) =\overline{T_1(\lambda-\I 0)}$ and 
$R_1(\lambda+\I 0) =\overline{R_1(\lambda-\I 0)}$ for $k_1(\la)\in\R$.
\vskip 1mm
\item[{\bf (b)}] $\dfrac{T_1(\lambda)}{\overline{T_1(\lambda)}}= R_1(\lambda)$ for $k_1(\la)\in [-c,c]$.
\item[{\bf (c)}] $1 - |R(\lambda)|^2 =
\dfrac{k_1}{k}\,|T(\lambda)|^2$ and $1 - |R_1(\lambda)|^2 =
\dfrac{k}{k_1}\,|T_1(\lambda)|^2$ for $k(\lambda)\in\R$.
\item[{\bf (d)}] $\overline{R(\lambda)}T(\lambda) +
R_1(\lambda)\overline{T(\lambda)}=\overline{R_1(\lambda)}T_1(\lambda) +
R(\lambda)\overline{T_1(\lambda)}=0$ for $k(\lambda)\in \R$.
\vskip 1mm
\item[{\bf (e)}] $T(\lambda) = 1 +O(\frac{1}{\sqrt\la})$, $T_1(\lambda) = 1 +O(\frac{1}{\sqrt\la})$  as $\lambda\to\infty$, \\
 $R(\lambda) = O(\frac{1}{\sqrt\la})$, $R_1(\lambda) = O(\frac{1}{\sqrt\la})$ as $\lambda\to\infty$.
\vskip 1mm
\item[{\bf II. \ }]
The functions $T(\lambda)$ and $T_1(\la)$ can be analytically continued to 
$\mathcal D$ satisfying
\[ 
2\I k(\la) T^{-1}(\la)  = 2\I k_1(\la) T_1^{-1}(\la) =:W(\la),
\]
where $W(\la)$ has the following properties: \begin{itemize}
\item[$(i)$] It is holomorphic in the domain $\mathcal D$ and continuous up to the boundary. Moreover, $W(\la + \I 0)=\ol{W(\la - \I 0)}\neq 0$ as $\la \in (-c^2, \infty)$.
\item[$(ii)$] In $\mathcal D$ its only zeros are $\la_1, \dots, \la_N$, and
\[ 
\left(\frac{d W}{d \la}\,(\la_j)\right)^{-2}=\gamma_j^2\gamma_{j,1}^2.
\]
\item[$(iii)$] If $W(-c^2)=0$ then 
\[ 
W(\la)=\I\gamma\sqrt{\la +c^2}(1+o(1)), \quad \mbox{as}\ \ \la\to -c^2,\quad \mbox{where}\ \ \gamma\in\R\setminus\{0\}.
\]
\end{itemize}
\item[{\bf III. }] $R(\la)$  and $R_1(\la)$ are continuous for $k(\la)\in\R$ and $k_1(\la)\in\R$, respectively. 
\end{itemize}
\end{lemma}
Define
\[K(x,y)=\frac{1}{2} B\Big(x, \frac{y-x}{2}\Big), \quad K_1(x,y)=\frac{1}{2} B_1\Big(x, \frac{y-x}{2}\Big),\]
and
\be\label{defF}
\aligned F(x)&=\frac{1}{2\pi}\int_{\R} R(\la)\E^{\I k x} dk +\frac{1}{4\pi}\int_{-c^2}^0 |T_1(\la)|^2\E^{\I k x}\frac{d\la}{|k_1|} +\sum_{j=1}^N\gamma_j^2\E^{-\kappa_j x};\\
F_1(x)&=\frac{1}{2\pi}\int_{\R} R_1(\la)\E^{-\I k_1 x}d k_1 +\sum_{j=1}^N\gamma_{j,1}^2 \E^{\kappa_{j,1} x},\quad \kappa_{j,1}=\sqrt{-c^2-\la_j}>0.\endaligned\ee
These functions are connected by the right and left Marchenko equations (\cite{BF, CK})
\be\label{March}
\aligned K(x,y) + F(x+y) + \int_x^{+\infty} K(x,s) F(y+s)ds&=0,\quad y>x;\\
K_1(x,y) + F_1(x+y) + \int^x_{-\infty} K_1(x,s) F_1(y+s)ds&=0,\quad y<x.
\endaligned\ee

\begin{lemma}[\cite{M}] \label{lemma 2.8} Let $q(x)$ satisfy \eqref{inithor} with $m_0\geq 1$ and $n_0\geq 0$. Then 
\vskip 2mm
{\bf IV.}\ $F(x), F_1(x)\in C^{n_0+1}(\R)$ and
\[ \label{dec} 
\int_0^{+\infty} \left|\frac{d^n F}{d x^n}\right|(1 +|x|^{m_0}) dx<\infty, \quad 
\int^0_{-\infty} \left|\frac{d^n F_1}{d x^n}\right|(1 +|x|^{m_0}) dx<\infty, \quad 1\leq n\leq n_0+1.
\]
\end{lemma}
The set of scattering data for the operator $L$ can be defined as follows:
\be\label{scatdata}\aligned
\mathcal S(m_0, n_0):=&\big\{R(\la), T(\la), k\in\R;\  R_1(\la), T_1(\la), k_1\in \R; \\ & \quad \la_1, \dots,\la_N\in (-\infty,-c^2),\ \gamma_1^2, \gamma_{1,1}^2, \dots, \gamma_N^2, \gamma_{N,1}^2\big\}.
\endaligned\ee
The functions $F(x)$ and $F_1(x)$ are uniquely defined by this set. 
\begin{theorem}[\cite{EGLT}] \label{charak}
Properties {\bf I}--{\bf IV} of Lemmas~\ref{1-3} and \ref{lemma 2.8} are necessary and sufficient for 
$\mathcal S(m_0, n_0)$ to be the set of scattering data for the Schr\"odinger operator $L$ with potential $q(x)$ satisfying \eqref{inithor} for any $m_0\geq 1$ and $n_0\geq 0$.
\end{theorem}
The sufficiency of these properties is established as follows. Given arbitrary $m_0\geq 1$ and $n_0\geq 0$ and the set $\mathcal S(m_0, n_0)$ consisting of four functions and 3N numbers as in \eqref{scatdata}, assume that this set satisfies properties {\bf I} to {\bf IV}. Then the Marchenko equations \eqref{March} are uniquely solvable (\cite{M}) for the kernels of the transformation operators $K$ and $K_1$. Moreover, the functions (cf.\ \eqref{form1})
\[ q_+(x)=-2\frac{d}{dx} K(x,x)=-\frac{d}{dx} B(x,0), \quad q_-(x)=2\frac{d}{dx} K_1(x,x)=\frac {d}{dx} B_1(x,0),\]
are such that $q_\pm\in C^{n_0}(\R)$ and (\cite{M})
\[\int_{\R_\pm}\left|\frac{d^n}{dx^n} q_\pm(x)\right|(1+|x|^{m_0})dx<\infty, \quad 0\leq n\leq n_0.\]
The last and most important step establishes that $q_+(x)=q_-(x) - c^2:=q(x)$ (cf.\ \cite{EGLT}), that is, $q(x)$ satisfies \eqref{inithor}. Moreover, the  operator $L$ with potential $q(x)$ reconstructed by use of the Marchenko equations, 
has the chosen set $\mathcal S(m_0, n_0)$ as the set of scattering data.
This scattering theory forms the basis of the inverse scattering transform which we apply in the next section to prove Theorem \ref{thoern}.

\section{The Cauchy problem for the KdV equation with steplike initial data of the class (\ref{inithor}), (\ref{estbig})}
We first recall some well-known facts about the Lax pair. By a classical solution of the KdV equation we mean a solution that has $3$ continuous derivatives with respect to $x$ and one with respect to $t$. Let $v(x,t)$ be such a solution. Introduce the Lax operators
\[L_v(t)=-\pa_x^2 + v(x,t), \quad P_v(t)=-4\pa_x^3 + 6 v(x,t)\pa_x + 3v_x(x,t).\] 
As is known (\cite{Lax}), the KdV equation is equivalent to the Lax equation
\[\pa_t L_v(t)=[P_v(t),\ L_v(t)]\]
considered in $H^5(\R)$. It implies the unique solvability  of the compatibility system
\be\label{comp} L_v(t)u=\la u, \quad u_t= P_v(t)u,\ee for any initial condition $u(\la,0,0)$ and $u_x(\la, 0,0)$. In turn,  if $u_1$ and $u_2$ are two solutions of \eqref{comp}, then their Wronskian does not depend either on $x$ or on $t$. Let $\psi_\pm(\la, x,t)$  be two Weyl solutions of the equation $L_v(t)u=\la u$ normalized by the condition $\psi_\pm(\la,0,t)=1$ and let $m_\pm(\la,t)=\frac{\pa}{\pa x}\psi_\pm(\la,0,t)$ be the Weyl functions.
\begin{lemma}[\cite{EGT}] Set
\[
\alpha_\pm(\la,t)=\exp\left(\int_0^t\Big(2\big(v(0,s) + 2\la\big)m_\pm(\la,s) - v_x(0,s)\Big)ds\right).
\]
Then the functions $u_\pm(\la,x,t)=\alpha_\pm(\la,t)\psi_\pm(\la,x,t)$ solve  \eqref{comp}.
\end{lemma}
In particular,  $u(\la,x,t)=\E^{\I k x + 4 \I k^3 t}$ solves \eqref{comp} with $v(x,t)=0$ and $u_1(\la,x,t)=\E^{-\I k_1 x - 4 \I k^3_1 t + 6\I c^2 k_1 t}$ solves \eqref{comp} with $v(x,t)=-c^2$.
For these solutions, 
\be\label{al}
\alpha_+(\la,t)=:\alpha(\la,t)=\E^{4\I k^3 t}, \quad \alpha_-(\la,t)=:\alpha_1(\la,t)=\E^{- 4 \I k^3_1 t + 6\I c^2k_1 t}.
\ee
These properties allow us to derive the evolution  of the scattering data with respect to time. Indeed,
assume that there exists a classical solution $q(x,t)$ of the Cauchy problem \eqref{kdv}, \eqref{inithor} for some 
pair $(m_0, n_0)$ and assume that $q(x,t)$ has at least a finite first moment of perturbation, that is,
\be\label{thor1}\int_0^{+\infty} (1 +|x|)\left(\left|\frac{\pa^n}{\pa x^n}q(x,t)\right| +\left|\frac{\pa^n}{\pa x^n}(q(-x,t)+c^2)\right|\right)dx<\infty,\quad  0\leq n\leq 3.\ee
Then we can apply the results of the previous section to the time-dependent Schr\"{o}\-dinger operator $L_q(t)$. In particular, one can construct the Jost solutions of the equation $-\pa_x^2 y + q(x,t)y=\la y$ by 
\be\label{jostdep}\aligned
\phi(\la,x,t)&=\E^{\I k x} + \int_x^{+\infty}K(x,y,t)\E^{\I k y} dy,\\
\phi_1(\la,x,t)&=\E^{-\I k_1 x} + \int^x_{-\infty}K_1(x,y,t)\E^{-\I k_1 y} dy,\endaligned\ee
and introduce their time-dependent Wronskian $W(\la,t)$ and the scattering relations,
\[\aligned T(\la,t)\phi_1(\la, x,t)&=\ol{\phi(\la,x,t)} + R(\la,t)\phi(\la,x,t),\quad k\in\R,\\
T_1(\la,t)\phi(\la, x,t)&=\ol{\phi_1(\la,x,t)} + R_1(\la,t)\phi_1(\la,x,t),\quad k_1\in\R.
\endaligned\]
The left/right normalizing constants now also depend on $t$ and are given by
\[\gamma_j^2(t):=\frac{1}{\int_{\R}\phi^2(\la_j,x,t) dx},\quad \gamma_{j,1}^2(t):=\frac{1}{\int_{\R}\phi_1^2(\la_j,x,t) dx}.\]
Since $\alpha(\la,0)=\alpha_1(\la,0)=1$  by \eqref{al}, it follows from the considerations above that 
\be\label{dynn}\aligned W(\la,t)&=\frac{W(\la,0)}{\alpha(\la,t)\alpha_1(\la,t)},\quad T_1(\la,t)=T_1(\la,0)\alpha(\la,t)\alpha_1(\la,t), \\
R_1(\la,t)&=R_1(\la,0)\alpha_1^2(\la,t),\quad R(\la,t)=R(\la,0)\alpha^2(\la,t),\endaligned
\ee 
and the time evolution of the scattering data is therefore given by (cf.\ \cite{Kh, EGT})
\begin{equation}
\begin{split}
\label{evol}
R(\la,t)&=R(\la)\E^{8\I k^3 t}, \\
T(\la,t)&= T(\la)\E^{4\I k^3 t- 4 \I k^3_1 t + 6\I c^2k_1 t}, \\
\gamma_j^2(t)&=\gamma_j^2 \E^{8 \kappa_j^3 t},
\end{split}
\quad 
\begin{split}
R_1(\la,t)& = R_1(\la)\E^{- 8 \I k^3_1 t + 12\I c^2k_1 t}, \\ 
T_1(\la,t)&=T_1(\la)\E^{4\I k^3 t- 4 \I k^3_1 t + 6\I c^2k_1 t}, \\ 
\gamma_{j,1}(t)&=\gamma_{j,1}\E^{-(8\kappa_{j,1}^3 + 12 c^2\kappa_{j,1})t}.
\end{split}
\end{equation}
Here we have abbreviated $R(\lambda)=R(\lambda,0)$, etc.
On the spectrum of multiplicity one, 
\[ |T_1(\la,t)|^2 = |T_1(\la)|^2\E^{8\I k^3 t}, \quad \la\in [-c^2,0]\pm\I 0.
\]
Therefore the time-dependent Marchenko equations have the form
\be\label{March1}
\aligned K(x,y,t) + F(x+y,t) + \int_x^{+\infty} K(x,s,t) F(y+s,t)ds&=0,\quad y>x,\\
K_1(x,y,t) + F_1(x+y,t) + \int^x_{-\infty} K_1(x,s,t) F_1(y+s,t)ds&=0,\quad y<x,\endaligned
\ee
where
\be\label{defF1}
\aligned F(x,t)=&\frac{1}{2\pi}\int_{\R} R(\la)\E^{\I k x+8\I k^3 t} dk +\frac{1}{2\pi}\int_{-c^2}^0 \frac{k|T_1(\la)|^2}{|k_1|}\E^{\I k x+8\I k^3 t}d k \\
&+\sum_{j=1}^N\gamma_j^2\E^{-\kappa_j x+8\kappa_j^3 t},\\
F_1(x,t)=&\frac{1}{2\pi}\int_{\R} R_1(\la)\E^{-\I k_1 x-8\I k_1^3 t + 12\I c^2 k_1 t}d k_1 
+\sum_{j=1}^N\gamma_{j,1}^2 \E^{\kappa_{j,1} x-8\kappa_{j,1}^3t -12 c^2\kappa_{j,1}t}.
\endaligned
\ee
Formulas \eqref{evol}--\eqref{defF1} were obtained assuming that $q(x,t)$ exists and tends to the background 
constants as in \eqref{thor1}. To prove that $q(x,t)$ indeed exists and is unique we have to check that the set
\be\label{scatdatat}\aligned
\mathcal S(t):=&\big\{R(\la,t), T(\la,t), k\in\R;\  R_1(\la,t), T_1(\la,t), k_1\in \R; \\ & \quad
\la_1, \dots, \la_N\in (-\infty,-c^2), \gamma_1^2(t), \gamma_{1,1}^2(t), \dots, \gamma_N^2(t), \gamma_{N,1}^2(t)\big\}
\endaligned\ee
defined via \eqref{evol}--\eqref{defF1} and \eqref{scatdata} (corresponding to the initial data) satisfies the 
necessary and sufficient conditions {\bf I}--{\bf IV}. It may happen that  $\mathcal S(t)$ will have less moments and derivatives 
than the initial data, that is, $\mathcal S(t)=\mathcal S(m_0(t), n_0(t))$ with $1\geq m_0(t)\geq m_0$, $3\geq n_0(t)\geq n_0$. This however will still guarantee unique solvability of \eqref{March1} and the equality $q(x,t)\equiv q_1(x,t)$, 
where
\[ q(x,t)=-2\frac{d}{dx}K(x,x,t),\quad q_1(x,t)=2\frac{d}{dx} K_1(x,x,t) - c^2, \quad x\in\R,
\]
as well as the required decay of these functions to their backgrounds.
Let us first check properties {\bf I}--{\bf III}.  Properties {\bf I.(a), (b)} follow immediately from 
\be\label{zum}
\aligned\alpha(\la+\I 0, t)& =\ol{\alpha(\la-\I 0, t)}=\alpha^{-1}(\la+\I 0, t), \quad  \la\in [0,\infty);\\ \alpha_1(\la+\I 0, t) &=\ol{\alpha_1(\la-\I 0, t)}=\alpha^{-1}_1(\la+\I 0, t), \quad \la\in[-c^2, \infty);\\ 
 \alpha(\la-\I 0, t)& =\alpha(\la+\I 0, t)\in\R, \quad \la\in [-c^2,0].\endaligned\ee
Since $|\alpha(\la,t)|=|\alpha_1(\la,t)|=1$ as $\la\in\R_+$, then the moduli of the entries of the scattering matrix do not depend on time on the spectrum of multiplicity 2, i.e. when $k\in\R$. This proves {\bf I.(c)}. Property {\bf I.(d)} follows from \eqref{dynn} and \eqref{zum}. Finally, for $k\to\infty$
\begin{align*}
\log(\alpha(\la,t)\alpha_1(\la,t))&=4\I (k^3 - k_1^3)t +6\I c^2k_1 t=O(k^{-1}),\\ 
\alpha(\la,t)\alpha_1(\la,t)&= 1+O(k^{-1}),
\end{align*}
which proves {\bf I.(e)}.
To check {\bf II},  note that $\alpha(\la,t)$ and $\alpha_1(\la,t)$ are holomorphic in $\C\setminus [-c^2, \infty)$ and continuous up to the boundary. This proves $(i)$. For $(ii)$ we use 
\[\frac{\pa W(\la,t)}{\pa \la}\Big|_{\la=\la_j}=(\alpha(\la_j,t)\alpha_1(\la_j,t))^{-1}\frac{d W(\la)}{d\la}\Big|_{\la=\la_j},
\] 
which is true since $W(\la_j)=0$. Together with \eqref{evol} this impies $(ii)$,
\[
\left(\frac{\pa W}{\pa \la}\,(\la_j,t)\right)^{-2}=\gamma_j^2(t)\gamma_{j,1}^2(t).
\]
Property {\bf II}, $(iii)$ follows from $\alpha(-c^2,t)\alpha_1(-c^2,t)\in\R\setminus\{0\}$.

The main difficulty in the application of IST to steplike solutions arises in the verification 
of property {\bf IV} and in finding optimal constants $m_0(t), n_0(t)$, which actually do not depend on $t$ as we will see. 

First of all, to differentiate $F(x,t)$ and $F_1(x,t)$ with respect to $x$ we need more decay of the reflection coefficients than established in 
{\bf I.(e)} of Lemma \ref{1-3}. But $R(\la)$ and $R_1(\la)$ are the reflection coefficients of the potential satisfying \eqref{inithor}, \eqref{estbig} and 
for such a potential it was proven in \cite{EGLT} that 
as $\la \to \infty$,
\[
 \frac{d}{dk}R(\la)=g_{+,s}(\la)\la^{-\frac{n_0+1}{2}}, \quad  \frac{d}{dk_1}R_1(\la)=g_{-,s}(\la)\la^{-\frac{n_0+1}{2}},\quad s=0, \dots,m_0-1,
 \]
 where $\sqrt\la\cdot g_\pm(\la)\in L^2 (a, +\infty)$ and $a$ is a sufficiently large positive number.

Let us examine the behavior of $F(x, t)$ as $x \rightarrow +\infty$ and $F_1(x, t)$ as $x \rightarrow -\infty$.
The summands which correspond to the discrete spectrum are exponentially small with respect to large 
$x$ for a fixed $t$, so they do not contribute to the behavior of the kernels. 

\subsection{Asymptotic behavior of the right kernel $F(x, t)$.}
Denote the two summands of $F(x, t)$ in \eqref{defF1} containing the integrals by $F_R(x, t)$ and $F_T(x, t)$. 
Then $F_T(x, t)$ can be rewritten as
\[
F_T(x, t) = \dfrac{1}{4\pi} \int_{-c^2}^0 |T_1(\la)|^2 k_1^{-1} e^{8ik^3t}e^{ikx}d\la = \dfrac{1}{\pi}\int_c^0 P(h) \psi(h) e^{-hx} dh,
\]
where $\sqrt{\la} = k = \I h$, $\psi(h)=e^{8h^3t}$ and
\begin{equation} \label{eq3.19}
P(h) = \dfrac{ \I h W(\phi_1(\la, \cdot, 0), \ol{\phi_1(\la, \cdot, 0)})}{W(\phi_1(\la, \cdot, 0), \phi(\la, \cdot, 0)) W(\ol{\phi_1(\la, \cdot, 0)}, \phi(\la, \cdot, 0))}.
\end{equation}
Using the symmetry property $R(\la(k))=\ol{R(\la(-k))}$ we can write
\begin{equation*}
F_R(x, t) = \dfrac{1}{\pi} \re \int_0^{+\infty}R(k)\psi(- \I k) e^{\I kx}dk,
\end{equation*}
where $R(k): = R(\la(k), 0)$ and $k \geq 0$.
Note that $P(h)$ does not depend on $t$, and $P(h)<0$ when $h \in (0, c)$. Since time $t$ is fixed, we omit $t$ in the notation of $\psi(h)$.
Let us abbreviate $\psi_s^{(j)}=\frac{\partial^j}{\partial s^j} \psi(h(s), t)$ and $\psi_s' = \frac{\partial}{\partial s} \psi(h(s), t)$.
Evidently,
\begin{equation} \label{eq3.20}
\psi_h^{(j)}(0) = \I^j \psi_k^{(j)}(0) \in \R.
\end{equation}
Integrating $F_R(x, t)$ by parts yields
\begin{equation*}
\begin{aligned}
\re \int_0^{\infty} R(k)\psi(-\I k)e^{\I kx}dk = \re \bigg\{-\dfrac{1}{\I x}R(0)\psi(0) + \dfrac{1}{(\I x)^2}(R_k' \psi + R \psi_k^\prime)(0) \\
+ \dots +  \frac{(-1)^m}{(ix)^m}\dfrac{\partial^{m-1}(R\psi)}{\partial k^{m-1}}(0) + \frac{(-1)^m}{(\I x)^m} \int_0^{\infty}\dfrac{\partial^{m}(R\psi)}{\partial k^{m}}e^{\I kx}dk \bigg\}.
\end{aligned}
\end{equation*}
We split $F_T(x, t)$ into a sum of two integrals, $\int_0^{\frac{c}{2}}  + \int_{\frac{c}{2}}^c$.
The second integral is evaluated as $O(e^{-\frac{c}{2}x})$, and this estimate can be differentiated with respect to 
$x$, which is sufficient for our purpose. On the other hand, $\phi_1(\la,x,t)$ does not have bounded derivatives with respect to $h$ at $h=c$, therefore $|T_1(\la,t)|^2$ is not differentiable too. Integrating  the first integral by parts yields for 
$\pi F_T(x, t)$
\begin{equation*}
\begin{aligned}
\int_c^0 P(h) \psi(h) \E^{-hx} dh &= - \frac{1}{x}P(0)\psi(0) - \frac{1}{x^2}(P_h'\psi + P\psi_h')(0) - \dots  
\\
& \quad - \frac{1}{x^m}\dfrac{\partial^{m-1}(P\psi)}{\partial h^{m-1}}(0) + \dfrac{1}{x^m} \int_{\frac{c}{2}}^0 \dfrac{\partial^{m}(P\psi)}{\partial h^{m}}e^{-hx}dh + O(\E^{-\frac{c}{2}x}),
\end{aligned}
\end{equation*}
where the term $O(e^{-\frac{c}{2}x})$ contains the second integral and integrands of the first integral 
corresponding to $h = \frac{c}{2}$.

Our next aim is to prove that
\begin{equation} \label{eq3.21}
\lim_{k\rightarrow +0} \bigg\{ \I^{j+1} \dfrac{\partial^{j}(R\psi)}{\partial k^{j}}(k) \bigg\} = \lim_{h\rightarrow+0} \dfrac{\partial^{j}(P\psi)}{\partial k^{j}}(h).
\end{equation}
The limit on the right side is taken from the side of the spectrum of multiplicity one. 
Since $\psi \in C^{\infty}(\R)$, using \eqref{eq3.20} we get
\begin{equation*}
\dfrac{\partial^{j}(P\psi)}{\partial k^{j}}(+0) = \sum_{s=0}^j C_j^s P_h^{(j-s)}(+0)\I^s \psi_k^{(s)}(0).
\end{equation*}
Here $C_j^s$ are the binomial coefficients. Compare this formula with 
\begin{equation*}
\re \bigg\{ \I^{j+1} \dfrac{\partial^{j}(R\psi)}{\partial k^{j}}(+0) \bigg\} = 
\sum_{s=0}^j C_j^s \I^{j+1-s} R_k^{(j-s)}(+0)\I^s \psi_k^{(s)}(0),
\end{equation*}
and taking into account that $\I^s \psi_k^{(s)}(0) \in \R$, we see that  (\ref{eq3.21}) is equivalent to
\begin{equation} \label{eq3.22}
\lim_{k\rightarrow +0} \re\bigg\{ \I^{l+1}R_k^{(l)}(k)\bigg\} = \lim_{h\rightarrow+0} P_h^{(l)}(h).
\end{equation}
To prove this, recall that $\phi_1(x, \la, 0)$ and $\overline{\phi_1(x, \la, 0)}$ are $m_0$-times differentiable by $k$ in a neighborhood of the point $k=0$. The function
 $\phi(x, \la, 0)$ has $m_0-1$ derivatives in $k$ at $k = 0$. We define $P_1(k) = P(-\I k)$, where $P$ is determined by (\ref{eq3.19}). The function $P_1(k)$ is defined on the right side of $[0, ic]$, but can be redefined by the same formula to the right of the point $k=0$. Moreover, it has $m_0-1$ continuous derivatives at $k=0$, namely
\begin{equation*}
\lim_{k\rightarrow+0} \I^l \dfrac{d^l P_1(k)}{dk^l} = \lim_{h\rightarrow+0} \dfrac{d^lP}{dh^l}.
\end{equation*}
Now we extend $P_1$ for $k\geq 0$. We can see that
\begin{equation*}
P_1(k) = \dfrac{\I W(\phi_1, \ol{\phi_1})(-2\I k)}{2W(\phi_1, \phi) W(\ol{\phi_1}, \phi)} = \dfrac{\I W(\phi_1, \ol{\phi_1})W(\phi, \overline{\phi})}{2W(\phi_1, \phi) W(\ol{\phi_1}, \phi)}.
\end{equation*}
Substituting
\begin{equation*}
\phi = \dfrac{W(\phi, \overline{\phi_1)}}{W(\phi_1, \overline{\phi_1})}\phi_1 + \dfrac{W(\phi, \phi_1)}{W(\phi_1, \overline{\phi_1})} \ol{\phi_1}
\end{equation*}
in the numerator of the previous fraction we get
\begin{equation} \label{eq3.23}
P_1(k) = \dfrac{\I}{2} \left( -\dfrac{W(\phi_1, \ol{\phi})}{W(\phi_1, \phi)} + \dfrac{W(\ol{\phi_1}, \ol{\phi})} {W(\ol{\phi_1}, \phi)} \right)=
\dfrac{\I}{2} \left( -\dfrac{V(k)}{W(k)} + \frac{\ol{W(k)}}{\ol{V(k)}}\right),
\end{equation}
where
\begin{equation*}
W(k) = W(\phi_1, \phi), \qquad
V(k) = W(\phi_1, \ol{\phi}).
\end{equation*}
This representation is only possible for $k \geq 0$ because $\ol{\phi(\la, x, 0)}$ cannot be continued to $[0, \I c]$.

\begin{lemma}
The following equality is valid,
\begin{equation} \label{eq3.24}
V_k^{(s)} (0) = (-1)^s W_k^{(s)} (0), \qquad  s= 0, \dots, m_0 -1.
\end{equation}
\end{lemma}
\begin{proof}
Denote by  $f_n(x)$  any function such that $f_n(x) \in C^{m_0-1}[0, \epsilon)$.
Since
\begin{equation*}
\phi(\la, x) + \ol{\phi(\la, x)} = 2 \cos(kx) + 2\int_x^{\infty} K(x, y)\,\cos (ky) dy,
\end{equation*}
then
\begin{equation*}
\phi(\la, x) + \ol{\phi(\la, x)} = f_1(k^2), \quad \dfrac{\partial}{\partial x}(\phi(\la, x) + \ol{\phi(\la, x)}) = f_2(k^2).
\end{equation*}
Also $\phi_1(\la, 0)=f_3(k^2)$ and $\tfrac{\partial}{\partial x}\phi_1(\la, 0)=f_4(k^2)$.
On the other hand,
\begin{equation*}
\phi(\la, x) - \ol{\phi(\la, x)} = 2\I k \left( \dfrac{\sin kx}{k} + \int_x^{\infty} K(x, y) \dfrac{\sin ky}{k} dy\right),
\end{equation*}
thus
\begin{equation*}
\phi(\la, x) - \ol{\phi(\la, x)} = kf_5(k^2), \quad \dfrac{\partial}{\partial x}(\phi(\la, x) - \ol{\phi(\la, x)}) = kf_6(k^2).
\end{equation*}
Therefore
\begin{equation*}
V(k) + W(k) = f_7(k^2), \qquad V(k) - W(k) = f_8(k^2).
\end{equation*}
From the last two equalities it follows that
\begin{equation*}
W_k^{(2s)}(k) -V_k^{(2s)}(k) =k f_{2s+9}(k^2), \quad V_k^{(2s+1)}(k) + W_k^{(2s+1)}(k) =k f_{2s+10}(k^2),
\end{equation*}
which proves (\ref{eq3.24}).
\end{proof}
Denote $w_s:=\lim_{k\rightarrow+0}\tfrac{d^sW(k)}{dk^s}$ and let
\begin{equation*}
W_{m_0-1}(k) = w_0 + w_1k + \dfrac{w_2}{2!}k^2 + \dots + \dfrac{w_{m_0-1}}{(m_0-1)!}k^{m_0-1}
\end{equation*}
be the Taylor polynomial for $W(k)$. From the previous lemma it follows that
\begin{equation*}
V(k) = W_{m_0-1}(-k) + o(k^{m_0-1}).
\end{equation*}
Since $k=0$ is an interior point of the spectrum, implying that $W(0) \neq 0$, and since by definition $R(k) = -\tfrac{V(k)}{W(k)}$, we see that
\begin{equation*}
R^{-1}(k) = R_{m_0-1}(-k) + o(k^{m_0-1}),
\end{equation*}
where $R_{m_0-1}(k)$ is a Taylor polynomial for $R(k)$ of degree $m_0-1$. Using  \eqref{eq3.23} we get
\begin{equation*}
P_1(k) = \dfrac{\I}{2} \left( R_{m_0-1}(k) - \ol{R_{m_0-1}(-k)} \right) + o(k^{m_0-1}),
\end{equation*}
which proves \eqref{eq3.22}.

The next step consists of verifying how often $F_R$ and $F_T$ can be integrated by parts.
If the initial data belongs to the class $\eqref{inithor}$, the first integral of $F_T$ can be integrated by parts $m_0-1$ times,
\begin{equation*}
F_T(x, t) = A(x, t) + O(e^{-\frac{c}{2}x}) + \dfrac{1}{x^{s-1}}\int_c^0 \dfrac{\pa^{s-1} (P\phi)}{\pa h^{s-1}}e^{-hx}dh,
\end{equation*}
where $A(x, t)$ corresponds to integrands at $h=0$. We have proven already that it cancels out with the integrand terms of $F_R$ at  $k=0$. The function $P$ has only $s-1$ derivatives with respect to $h$ and this procedure stops. The  integral can be differentiated by $x$ any number of times.

Each time we integrate $F_R$ by parts, there appears a new multiplier $3\I k^2t$ in front of $R(k)$, since we 
differentiate $\psi(-\I k)=e^{\I k^3t}$. As the integrand function should remain summable in $L_2(\R_+)$, 
we require $k^{2s-1}R(k) \in L_2(\R_+)$.
If the initial data has enough derivatives, $F_R$ can be integrated $s-1$ times too. Since we are searching for 
the classical solution of KdV, it has to have at least three derivatives. Therefore we require 
$k^{2s+2}R(k) \in L_2(\R_+)$. In summary,
\begin{lemma}\label{got}
Let 
\[\int_0^{+\infty} (1 +|x|^{s})\left(\left|\frac{d^n}{d x^n}q(x)\right| +\left|\frac{d^n}{d x^n}(q(-x)+c^2)\right|\right)dx<\infty,
\quad  0\leq n\leq l, \ s\geq 3. 
\] 
Then for any $m \leq s-1$ and $l \geq 2m+3$, the function $F(x, t)$  can be represented as
\begin{equation*}
F(x, t) = \dfrac{D(x, t)}{x^m} + D_1(x,t),\quad x\rightarrow +\infty, 
\end{equation*}
where $\frac{\pa^j D(x, t)}{\pa x^j} \in L_2(a,+\infty)$,  $0 \leq j \leq l -2m+1$, and 
$\frac{\pa^i D_1(x, t)}{\pa x^i}=O(e^{-\frac{c}{2}x})$ for $i \geq 0$, uniformly for $t \in [-\mathcal{T}, \mathcal{T}]$.
\end{lemma}

\subsection{Asymptotic behavior of the left kernel $F_1(x, t)$ in \eqref{defF1}}

We discard terms corresponding to the discrete spectra because they are exponentially small. The right Jost solution and hence the Wronskians in the expressions for $T, T_1$ and $R, R_1$ are functions of the local parameter $\sqrt{k_1 \mp c}$. Thus, if we want to integrate by parts, we need to differentiate the left reflection coefficient by $k_1$, which leads to singularities of type
\begin{equation*}
\dfrac{\pa^s \mathcal{R}(k_1)}{\pa k_1^s} = O\left(\left(k_1\mp c\right)^{-\frac{2s-1}{2}}\right), \quad s\geq 1 \text{ and } k_1 \rightarrow \pm c,
\end{equation*}
where $\mathcal{R}(k_1)=R_1(\la, t)$. Let $0 < \varepsilon < \frac{c}{8}$ and introduce 
\begin{equation*}
\mathcal B_{\pm}(k_1):= \mathcal B\left(\dfrac{k_1 \mp c}{\varepsilon}\right), \quad \text{ where }\mathcal  B(\xi) =
\begin{cases}
(1 - \xi^{2m_0})^{2m_0}, &\text{if } |\xi| \leq 1,\\
0, &\text{if } |\xi| \geq 1.
\end{cases}
\end{equation*}
Obviously $\frac{d^s\mathcal B_{\pm}}{dk_1^s} (\pm c+ \varepsilon) = \frac{d^s\mathcal B_{\pm}}{dk_1^s} ( \pm c-\varepsilon)= 0$, and $\frac{d^s\mathcal B_{\pm}}{dk_1^s} (\pm c) = 0$  for $0\leq s \leq 2m_0-1$.

We rewrite the integral containing $R_1$ in  \eqref{defF1} as 
\[ 
\begin{aligned}
& \int_{\R} R_1(\la)e^{-8ik_1^3t+12ic^2k_1t}e^{-ik_1x}dk_1 = \int_{\R} \left(1 - \mathcal B_+(k_1) - \mathcal B_-(k_1) \right) \mathcal{R}(k_1)e^{-ik_1x}dk_1  \\
& \qquad + \int_{-c-\varepsilon}^{-c+\varepsilon} \mathcal B_-(k_1) \mathcal{R}(k_1)e^{-ik_1x}dk_1 + \int_{c-\varepsilon}^{c+\varepsilon } \mathcal B_+(k_1) \mathcal{R}(k_1)e^{-ik_1x}dk_1 \\
&\quad =: I_1(x,t) + I_2^-(x, t) + I_2^+(x,t).
\end{aligned}
\]

The function $1 - \mathcal B_+(k_1) - \mathcal B_-(k_1)$ has zeros of degree $2m_0-1$ at $k_1=c$ and $k_1 = -c$ and is infinitely differentiable. Since the behavior of $\frac{\pa^m}{\pa k_1^m} (\mathcal{R}(1 - \mathcal B_+ - \mathcal B_-))$ as $k_1 \rightarrow \infty$ is the same as that of $\frac{\pa^m \mathcal{R}(k_1)}{\pa k_1^m}$, $k_1 \rightarrow \infty$, the integral $I_1(x, t)$ 
can be integrated by parts similarly to our previous considerations. 
To evaluate $I_2^{\pm}$, we  focus on $I_2^+$, the evaluation of $I_2^-$ is done in the same way. We start with
the Taylor series for $\mathcal{R}(k_1)$ on the interval $[-\varepsilon+c, \varepsilon+c]$. Let $z^2 = k_1 - c$, then 
\begin{equation*}
\mathcal{R}(c + z^2) = a_0(t )+ a_1(t)z + a_2(t)z^3 + \dots + a_{m_0-1}(t)z^{m_0-1} + \beta(z, t),
\end{equation*}
and $\beta(z, t) = O(z^{m_0})$ as $z \rightarrow 0$. Thus, $\mathcal{R}$ has at least $[\frac{m_0}{2}]$ derivatives with respect to $k_1$ at the point $k_1 = c$. Since $\mathcal B_+$ and its derivatives disappear at $k_1 = -\varepsilon+c$ and $k_1 = \varepsilon + c$, integration by parts gives
\begin{equation*}
I_2^+(x, t) = e^{-\I cx} \sum_{j=0}^{m_0-1}a_j(t) \int_{-\varepsilon}^{\varepsilon} \zeta^{\frac{j}{2}} \mathcal B\left(\dfrac{\zeta}{\varepsilon}\right)e^{-\I x\zeta}d\zeta + \dfrac{h_+(x, t)}{x^{[\frac{m_0}{2}]}},
\end{equation*}
where $h_+(x, t) \in L_2(\R_-)$, it is infinitely differentiable with respect to $x$, and all of the derivatives are also in $L_2(\R_-)$.

Let $\xi = \frac{\zeta}{\varepsilon}$. Using $1 - \xi^{2m} = (1 - \xi^2)(1 +\xi^2 + \xi^{4} + \dots + \xi^{2m-2})$ we get
\begin{equation*}
\begin{aligned}
\left(I_2^+(x, t)  - \dfrac{h_+(x, t)}{x^{[\frac{m_0}{2}]}}\right)\E^{\I c x} &= \sum_{j=0}^{[\frac{m_0}{2}]} \tilde{a}_{2j}(\varepsilon, t) \int_{-1}^1 \xi^{j}\mathcal B(\xi) e^{-\I x \xi \varepsilon} d\xi 
\\
& + \sum_{j=1}^{2m_0-2+[\frac{m_0}{2}]} \tilde{a}_{2j-1}(\varepsilon, t) 
\int_{-1}^1 \xi^{\frac{2j-1}{2}} (1-\xi^2)^{2m_0} e^{-\I x \xi \varepsilon} d\xi.
\end{aligned}
\end{equation*}
Each term of the first sum can be integrated by parts at least $2m_0-1$ times, moreover, all  integrands are equal to zero. All the summands of the second sum, except for the first one,  can be integrated by parts until we get the integral 
\begin{equation} \label{eq3.26}
\begin{aligned}
&J\left(x, 2 m_{0}\right) = \int_{-1}^{1} \xi^{\frac{1}{2}}\left(1-\xi^{2}\right)^{2 m_{0}} \mathrm{e}^{-\mathrm{i} \xi x \varepsilon} d \xi= \frac{1}{2}\mathcal B(\nu, \rho)\ _1F_{2}\bigg(\nu, \frac{1}{2}, \nu+\rho, \frac{\mu^{2}}{4}\bigg)\\
&+\frac{\mu}{2} \mathcal B\bigg(\nu+\frac{1}{2}, \rho\bigg)\ _1F_{2}\bigg(\nu+\frac{1}{2}, \frac{3}{2}, \nu+\rho+\frac{1}{2}, \frac{\mu^{2}}{4}\bigg) +\frac{\I}{2} \mathcal B(\nu, \rho)\ _1F_{2}\bigg(\nu, \frac{1}{2}, \nu+\rho, \frac{\mu^{2}}{4}\bigg)\\
&-\frac{\mu \mathrm{i}}{2} \mathcal B\bigg(\nu+\frac{1}{2}, \rho\bigg)\ _1F_{2}\bigg(\nu+\frac{1}{2}, \frac{3}{2}, \nu+\rho+\frac{1}{2}, \frac{\mu^{2}}{4}\bigg),
\end{aligned}
\end{equation}
where $\mu = \I|x\varepsilon|$, $2\nu-1 = \frac{1}{2}$ and $\rho = 2m_0+1$. Here
\begin{equation*}
\mathcal B(\nu, \rho)=\frac{\Gamma(\nu) \Gamma(\rho)}{\Gamma(\nu+\rho)}
\end{equation*}
is the beta function, $\Gamma(\gamma)$ is the gamma function and $_1F_{2}(a, b, c, -z)$ is the generalized hypergeometric function. Using its asymptotic behavior for large  $|z|$  we get (cf.\ \cite[Equ.\ 16.11.8]{Olver})
\begin{equation} \label{eq3.28}
_1F_{2}(a, b, c,-z)=\frac{\Gamma(b) \Gamma(c)}{\Gamma(a)}\left(H_{1,2}(z)+E_{1,2}\left(z \mathrm{e}^{-\pi \mathrm{i}}\right)+E_{1,2}\left(z \mathrm{e}^{\pi \mathrm{i}}\right)\right)
\end{equation}
where
\begin{equation} \label{eq3.29}
\begin{aligned}
H_{1,2}(z)&=\sum_{k=0}^{c-a-1} \frac{\Gamma(a+k)}{\Gamma(b-a-k) \Gamma(c-a-k)} z^{-a-k} \frac{(-1)^{k}}{k !},
\\
E_{1,2}(z)&=\frac{1}{\sqrt{2 \pi}} 2^{-\alpha-(1 / 2)} \mathrm{e}^{2 z^{1 / 2}} \sum_{n=0}^{\infty} c_{n}\left(2 z^{1 / 2}\right)^{\alpha-n}, \quad \alpha=a-b-c+\frac{1}{2},
\end{aligned}
\end{equation}
and $c_n$ are constants. In \eqref{eq3.26} we substitute $a=\nu$, $b=\frac{1}{2}$, $c=\nu+\rho$ or $a=\nu+\frac{1}{2}$,  $b=\frac{3}{2}$, $c=\nu+\rho+\frac{1}{2}$ and also $\frac{\mu^{2}}{4}=:-z$, $\frac{-\mu \mathrm{i}}{2}=\sqrt{z} \in \R_{+}$. Then $\left(z \mathrm{e}^{\pm \I}\right)^{1 / 2} \in \mathrm{i} \R$, that is, the module of the last two terms in \eqref{eq3.28} containing the exponent is equal to 1. On the other hand,  in (\ref{eq3.29}) either $\alpha = -\rho = -2m_0 -1$ or $\alpha = -\rho - 1$. In any case, the terms containing $E_{1,2}$ can be evaluated as $ \frac{1}{(\varepsilon x)^{m_{0}+1 / 2}}($const.$+o(1))$. Thus, we can discard these terms. Combining the first  and the last summand of (\ref{eq3.26}) (the two other summands give the same function up to the multiplier $\I$) yields
\begin{equation*}
\begin{aligned}
&\frac{1}{2} \mathcal B(\nu, \rho)\ _1F_{2}\bigg(\nu, \frac{1}{2}, \nu+\rho,-z\bigg)& \\
&\quad +z^{1 / 2}\mathcal B\bigg(\nu+\frac{1}{2}, \rho\bigg)\ _1F_{2}\bigg(\nu+\frac{1}{2}, \frac{3}{2}, \nu+\rho+\frac{1}{2},-z\bigg)=:
A_{1}+A_{2},
\\
&_1F_{2}\bigg(\nu, \frac{1}{2}, \nu+\rho,-z\bigg) \sim \frac{\Gamma\left(\frac{1}{2}\right) \Gamma(\nu+\rho)}{\Gamma(\nu)} \sum_{k=0}^{\rho-1} \frac{\Gamma(\nu+k)}{\Gamma\left(\frac{1}{2}-\nu-k\right) \Gamma(\rho-k)} z^{-\nu-k} \frac{(-1)^{k}}{k !},
\end{aligned}
\end{equation*}
that is,
\[
A_{1}=\frac{1}{2} \frac{\Gamma(\nu) \Gamma(\rho)}{\Gamma(\nu+\rho)} \frac{\Gamma\left(\frac{1}{2}\right) \Gamma(\nu+\rho)}{\Gamma(\nu)} \sum_{k=0}^{\rho-1} \frac{\Gamma(\nu+k)}{\Gamma\left(\frac{1}{2}-\nu-k\right) \Gamma(\rho-k)} z^{-\nu-k} \frac{(-1)^{k}}{k !}.
\]
Similarly,
\begin{equation*}
\begin{aligned}
A_{2} &=z^{1 / 2} \frac{\Gamma\left(\nu+\frac{1}{2}\right) \Gamma(\rho)}{\Gamma\left(\nu+\frac{1}{2}+\rho\right)} \frac{\Gamma\left(\frac{3}{2}\right) \Gamma\left(\nu+\frac{1}{2}+\rho\right)}{\Gamma\left(\nu+\frac{1}{2}\right)}
\\
& \qquad \times \sum_{k=0}^{\rho-1} \frac{\Gamma\left(\nu+\frac{1}{2}+k\right)}{\Gamma\left(\frac{3}{2}-\nu-\frac{1}{2}-k\right) \Gamma(\rho-k)} z^{-\nu-\frac{1}{2}-k} \frac{(-1)^{k}}{k !}.
\end{aligned}
\end{equation*}
Summing up these equalities and taking into account that $\nu = \frac{3}{4}$, we get
\begin{equation*}
\begin{aligned}
&A_{1} +A_{2}=\sum_{k=0}^{\rho-1} \frac{(-1)^{k}}{k !} z^{-\nu-k} \frac{\Gamma(\rho) \Gamma\left(\frac{3}{2}\right)}{\Gamma(\rho-k)}\left(\frac{\Gamma(\nu+k)}{\Gamma\left(\frac{1}{2}-\nu-k\right)}+\frac{\Gamma\left(\nu+\frac{1}{2}+k\right)}{\Gamma(1-\nu-k)}\right)+
\\
&+O\big(z^{-\nu-\rho+1}\big) \sum_{k=0}^{\rho-1} \frac{(-1)^{k}}{k !} z^{-\nu-k} \frac{\Gamma(\rho) \Gamma\left(\frac{3}{2}\right)}{\Gamma(\rho-k)}\left(\frac{\pi}{\sin \left(\frac{3 \pi}{4}+\pi k\right)}+\frac{\pi}{\sin \left(\frac{3 \pi}{4}+\frac{\pi}{2}+\pi k\right)}\right)
\\
&+O\big(z^{-\nu-\rho+1}\big)=O\big(x^{-2 m_{0}-1}\big).
\end{aligned}
\end{equation*}
Since $E_{1,2} (ze^{\pm i \pi})$ have the same order, we conclude that 
$J(x, 2m_0) = O(x^{-m_0 - \frac{1}{2}})$. Also
\begin{equation*}
\frac{d}{d x} J(x, 2 m_{0})=\frac{1}{\mathrm{i} \varepsilon x}\left(J(x, 2 m_{0})-4 m_{0} \int_{-1}^{1} \xi^{3 / 2}\left(1-\xi^{2}\right)^{2 m_{0}-1} \mathrm{e}^{-\mathrm{i} \xi(x \varepsilon)} d \xi\right),
\end{equation*}
and the last integral is evaluated as $O\big(x^{-2} J(x, 2 m_{0}-2)\big)$ (plus terms of higher order). Therefore, derivatives of $J(x, 2m_0)$ decrease even faster. Hence the main contribution to $I_2^{\pm}(x, t)$ comes from the term 
$h_{\pm}(x, t)x^{-[\frac{m_0}{2}]}$. That  is why we can integrate $I_1(x, t)$ by parts exactly $[\frac{m_0}{2}]$ times.
The behavior of the integrand as $k_1\to \pm\infty$ after the last integration is evaluated as $k_{1}^{2\left[\frac{m_{0}}{2}\right]-n_{0}-1} f\left(k_{1}\right) \mathrm{e}^{-\mathrm{i} k_{1} x}$, where $f \in L_2(\pm\infty)$. Every differentiation of $I_1(x, t)$ with respect to $x$ will increases the degree $2\left[\frac{m_{0}}{2}\right]-n_{0}-1 $ by 1. To obtain a classical solution we need to differentiate at least 4 times. Thus, we have to require $m_0 - n_0 - 1 \leq -4$, or $n_0\geq m_0 +3$. The following lemma sums up these considerations.
\begin{lemma}
Let $q(x)$ satisfy \eqref{inithor} with $n_0\geq m_0 +3$. Then $F_1(x, t)$ admits the representation as $x\rightarrow -\infty$
\[
F_{1}(x, t)=\frac{H_{1}(x, t)}{ x^{[\frac{m_0}{2}]}}, \quad \frac{\partial^{j} H_{1}(x, t)}{\partial x^{j}} \in L_{2}(-\infty, -a),
\quad j=0, \ldots, n_{0}-m_{0}+1,
\] 
with $a\gg 1$.
\end{lemma}

From Lemma \ref{got} we obtain
\begin{lemma}
Let $q(x)$ satisfy \eqref{inithor} with $n_0\geq m_0 +3$. Then $F(x, t)$ admits the representation
as $x\rightarrow +\infty$
\[
F(x, t)=\frac{H(x, t)}{ x^{[\frac{m_0}{2}]}}, \quad \frac{\partial^{j} H(x, t)}{\partial x^{j}} \in L_{2}(a, +\infty), 
\quad j=0, \ldots, n_{0}-m_{0}+1,
\]
with $a\gg 1$.
\end{lemma}
Since $\frac{1}{x} \frac{\partial^{j} H(x, t)}{\partial x^{j}} \in L_{1}(a,+\infty) $ and $ \frac{1}{x} \frac{\partial^{j} H_{1}(x, t)}{\partial x^{j}} \in L_{1}(-a, -\infty)$, both $F_x^{(j)}(x, t)$ and $F_{1, x}^{(j)}(x, t)$ satisfy condition {\bf IV} with $m_0(t) = [\frac{m_0}{2}]-1$ and $n_0(t) = n_0 - m_0$. Naturally, we assume that $m_0(t)\geq 1$, i.e., $m_0\geq 4$. This proves Theorem~\ref{thoern}. 

\section{Proof of Theorem \ref{maintheor1}.}
We represent the right Marchenko equation \eqref{March1} in the form (compare \cite{DT})
\[ B(x,y,t) + \hat F(x+y, t) +\int_0^\infty B(x,s,t)\hat F(x+y+s,t)ds=0,
\]
where
\[
B(x,y,t)=2K(x, x+2y,t),\quad \hat F(x,t)= 2F(2x,t).
\]
The kernel $\hat F(x,t)$ consists of three summands,
\begin{align*} \hat F_T(x,t)&=\dfrac{2}{\pi}\int_c^0 P(h) \E^{8 h^3 t -2 h x} dh, \quad
\hat F_d(x,t)= 2\sum_{j=1}^N\gamma_j^2\E^{-2\kappa_j x+8\kappa_j^3 t}, \\
\hat F_R(x,t)&=\dfrac{2}{\pi}\re\int_0^{+\infty} R(k)\E^{8\I k^3 t + 2 \I k x}dk, 
\end{align*}
where $k=\I h$ and $P(h)$ is defined by \eqref{eq3.19}. 
The following result is well known.
\begin{lemma}[\cite{Zakh}] \label{zak}  Let $\delta_{ij}$ be the Kronecker symbol and let $A(x,t)$ be a 
$N\times N$ matrix with elements
\[
A_{ij}(x,t)=\delta_{ij} +\frac{\gamma_j^2\E^{8\kappa_j^3 t}}{\kappa_i +\kappa_j}\E^{-(\kappa_j + \kappa_j)x}.
\]
Let $A^{(j)}(x,t)$ be the matrix obtained from $A(x,t)$ by replacing the $j$-th column of $A$ with 
\[
\begin{pmatrix}
-\gamma_1^2\E^{8\kappa_1^3 t -\kappa_1 x} \\
 \vdots \\
 -\gamma_N^2\E^{8\kappa_N^3 t -\kappa_N x}
\end{pmatrix}
\]
Then the Marchenko equation with kernel $\hat F_d(x,t)$,
\be\label{nuzh}
B_d(x,y,t) + \hat F_d(x+y, t) +\int_0^\infty B_d(x,s,t)\hat F_d(x+y+s)ds=0,
\ee
has a unique solution $B_d(x,y,t)$ such that
\[
B_d(x,0,t)=\frac{1}{\det A(x,t)}\sum_{j=1}^N \det A^{(j)}(x,t)\E^{-\kappa_j x}.
\]
The reflectionless, fast decaying solution $u(x,t)=-\frac{\pa B_d(x,0,t)}{\pa x}$ of the KdV equation associated with \eqref{nuzh} can be expressed as
\[
u(x,t)=-2\frac{\pa^2}{\pa x^2}\log\det A(x,t).
\]
In the domain $x>\varepsilon t$, the following asymptotic is valid as $t\to +\infty$ with $C>0$,
\be\label{solit} 
u(x,t)=-\sum_{j=1}^N\frac{2\kappa_j^2}{\cosh^2\left(\kappa_j x - 4\kappa_j^3 t -\frac{1}{2}\log\frac{\gamma_j^2}{2\kappa_j}-\sum_{i=j+1}^N\log\frac{\kappa_j - \kappa_i}{\kappa_i + \kappa_j}\right)} + O(\E^{-C t}).
\ee
\end{lemma}
Our aim is to prove that the main contribution to $\hat F(x,t)$ in the region 
\be\label{mathd}\mathcal O:=\left\{ (x,t): t\geq \mathcal T,\  x\geq 4c^2 t + \log t^{\frac{m_0-3/2 -\varepsilon}{c}}\right\},
\ee  
stems from $\hat F_d(x,t)$.  To this end we have to estimate $\hat F_T(x+y,t) +\hat F_R(x+y,t)$ for $(x,t)\in \mathcal O$, $y\geq 0$.
  Since $t$ is arbitrary large, we  cannot use 
integration by parts as in the previous section. Denote  
\[
\aligned x=4c^2 t& +\frac{m_0-\frac 3 2 -\varepsilon}{c}\log t + \xi,\quad  r=2(\xi+y),\\
  S(r,t,h)&:=8 h^3 t - 8hc^2 t -h\frac{2m_0-3-2\varepsilon}{c}\log t - h r.
  \endaligned\]
Note that for $t \geq \mathcal T$, $r\geq 0$, and $0\leq h\leq \frac{c}{2}$,
\[
\aligned S_h(r,t,h):=\frac{\pa S(r,t,h)}{\pa h}&=24h^2 t   -8c^2 t- 2\frac{m_0-\frac 3 2 -\varepsilon}{c}\log t - r<0, \\
\left|\frac{1}{S_h(r,t,h)}\right|&\leq \frac{1}{2c^2 t + r +1},
\endaligned
\] 
where $\mathcal T\geq\E^{\frac{c}{2m_0-3-2\varepsilon}}$.
 Split the integral in $\hat F_T(x+y,t)$ into 
two parts, 
\begin{align} \label{razbi}
F^{(1)}(x+y,t)&=\frac{2}{\pi}\int_c^{c/2} P(h) \E^{S(r,t,h)} dh, \\ \label{rzd}
F^{(2)}(x+y,t)&=\frac{2}{\pi}\int_{c/2}^0 P(h)\E^{S(r,t,h)}dh= \frac{2}{\pi}
\int_{c/2}^0 \frac{P(h)}{S_h(r,t,h)}d\E^{S(r,t,h)}.
\end{align}
Since
\[ S(r,t,h)\leq -\Big(m_0-\frac 3 2 -\varepsilon\Big)\log t -\frac{c}{2}r,\quad \mbox{for}\ h\in [c/2, c],
\]
then
\be\label{estkl}\aligned | F^{(1)}(x+y,t)|&\leq \frac{2}{\pi} \frac{1}{t^{m_0-\frac 3 2 -\varepsilon}}\E^{-\frac{c}{2} (\xi +y)}\int^c_{c/2}|P(h)|d h,\\
\left| \frac{\pa}{\pa x} F^{(1)}(x+y,t)\right|&\leq \frac{2}{\pi} \frac{1}{t^{m_0-\frac 3 2 -\varepsilon}}\E^{-\frac{c}{2} (\xi +y)}\int^c_{c/2}|h P(h)|d h.\endaligned
\ee 
On the other hand, we observe that $\psi(r,t,h):=\frac{1}{S_h(r,t,h)}$  satisfies \eqref{eq3.20}. An elementary analysis shows that 
\be\label{estpsi1} 
\left|\frac{\pa^n}{\pa h^n}\psi(r,t,h)\right|\leq C|\psi(r,t,h)| \leq 
C\begin{cases} \begin{array}{ll}\frac{1}{2c^2t + r +1}, & h\in [0, \frac{c}{2}],\\
\frac{1}{24 k^2 t + 8c^2 t + r + 1}, & k=\I h\in \R_+.\end{array}\end{cases}
\ee
The positive constant $C$ does not depend on $t$ and $r$. Let us simultaneously integrate both integrals in the sum 
\be\label{11}\aligned I(x+y, t):=& \frac{\pi}{2}\left(F^{(2)}(x+y,t) + \hat F_R(x+y,t)\right)\\
=&\int_{c/2}^0 P(h)\psi(r,t,h)d\E^{S(r,t,h)} +\re\int_0^\infty R(k)\psi(r,t,-\I k) d\E^{S(r,t,-\I k)} \endaligned\ee  by parts $m_0$-times, using each time the representation
\[\E^{S(r,t,h)}dh=\psi(r,t,h) d \E^{S(r,t,h)}, \quad -\I \E^{S(r,t,-\I k)}dk= \psi(r,t,-\I k) d \E^{S(r,t,-\I k)}.
\] 
Taking into account \eqref{eq3.20}, \eqref{eq3.22}, \eqref{estpsi1} we get 
\[ \aligned
I(x+y, t)&=\E^{S(r,t,c/2)} \mathcal P_1\Big(r,t,\frac{c}{2}\Big)
+ \int_c^{c/2} \mathcal P_2(r,t,h)\left(\psi(r,t,h)\right)^{m_0-1} \\ 
&\quad +\re\int_0^\infty \mathcal P_3(r,t,-\I k)
\big(\psi(r,t,-\I k)\big)^{m_0-1},
\endaligned\]
where the functions $\mathcal P_i(r,t,h)$ are uniformly bounded with respect to all arguments as $t\geq\mathcal T$, $r\geq 0$, $h\in [0, \frac c 2]$ and $k\in \R_+$.
Using again estimate \eqref{estpsi1} and inequality  $S(r,t,\frac c 2)< -7c^3 t - \frac{ c }{2} r$ we obtain
\be\label{estm1}\aligned |I(x+y,t)|& \leq C  \left(\E^{-7c^3t}\E^{-\frac{c r}{2}} + \int_{c/2}^c \frac{dh}{\left(2c^2 t + r + 1\right)^{m_0-1}}\right.\\
&\left. \qquad +\int_{\R_+}\frac{dk}{\left(24k^2 t + 8c^2 t + r +1\right)^{m_0-1}}\right).
\endaligned \ee

For the next step recall Young's inequality: For all $u>0, v>0$, $p>1$, $q>1$, such that $\frac{1}{p}+\frac{1}{q}=1$, one has  
\[ uv\leq \frac{u^p}{p} +\frac{v^{q}}{q}.
\]
In the last integral in \eqref{estm1} we choose 
 $u^p=p\cdot (24 k^2 t +8c^2 t)$, $v^q=q \cdot(r+1)$. In the first integral $v$ is the same and $u^p=p\cdot 2c^2 t$. 
Then
\be\label{13}\aligned  |I(x+y,t)|&\leq C\left(\E^{-7c^3t}\E^{-\frac{c r}{2}} +\frac{c}{2(2pc^2)^{\frac{m_0-1}{p}}\cdot q^{\frac{m_0-1}{q}}}\frac{1}{ t^{\frac{m_0-1}{p}} (r+1)^{\frac{m_0-1}{q}}}\right.
\\ &+\left.\frac{1}{p^{\frac{m_0-1}{p}}\cdot q^{\frac{m_0-1}{q}}} \frac{1}{t^{\frac{m_0-1}{p}} (r+1)^{\frac{m_0-1}{q}}}\int_{\R_+}\frac{d k}{ (24 k^2 + 8 c^2)^{\frac{m_0-1}{p}}}\right).\endaligned
\ee
To achieve convergence in the last integral we have to require $m_0-1>\frac{p}{2}$. In later 
estimates we need the property $(r+1)^{-\frac{m_0-1}{q}}\in L_2(\R_+)$, which implies $m_0-1>\frac{q}{2}$. 
Set
\[ \frac{1}{q}=\frac{\frac 1 2 +\varepsilon}{m_0-1}, \quad \ \frac 1 p= 1 -\frac{1}{q},\quad \mbox{that is,} \quad  \frac{m_0-1}{p}=m_0-\frac 3 2 - \varepsilon,
\]
where $\varepsilon>0$ is arbitrary small. Combining this with \eqref{11} and \eqref{13} implies
\be\label{nui}\big|F^{(2)}(x+y,t) + \hat F_R(x+y,t)\big|\leq \ti C\frac{1}{t^{m_0-\frac 3 2 -\varepsilon}}\ \frac{1}{(\xi + y+1)^{\frac 1 2 +\varepsilon}},\ee
where $\ti C>0$ does not depend on $x,y,t$ in the region 
\[t\geq\mathcal T, \quad x\geq 4c^2 t +\frac{m_0-\frac 3 2-\varepsilon}{c}\log t,\quad y\geq 0.
\] 
Repeating almost literally the arguments above  for $\frac{\pa}{\pa x}(F^{(2)}(x+y,t) + \hat F_R(x+y,t))$ 
we get the estimate
\be\label{nui1}
\left|\frac{\pa}{\pa \xi}(F^{(2)}(x+y,t) + \hat F_R(x+y,t))\right|\leq \hat C\frac{1}{t^{m_0-\frac 3 2 -\varepsilon}}\ \frac{1}{(\xi + y+1)^{\frac 1 2 +\varepsilon}}.
\ee
Combining  \eqref{razbi}--\eqref{estkl} with \eqref{nui}, \eqref{nui1}  we finally obtain that
\begin{align}\label{mainest} 
& \big|\hat F_T(x+y,t) + \hat F_R(x+y,t)\big|\leq C\frac{1}{t^{m_0-\frac 3 2 -\varepsilon}}\ \frac{1}{(\xi + y+1)^{\frac 1 2 +\varepsilon}}, \\
\label{mainest1}
& \left|\frac{\pa}{\pa \xi}\big(\hat F_T(x+y,t) + \hat F_R(x+y,t)\big)\right|\leq C\frac{1}{t^{m_0-\frac 3 2 -\varepsilon}}\ \frac{1}{(\xi + y+1)^{\frac 1 2 +\varepsilon}}.
\end{align}
 
 Consider the space of functions $\varphi(y) \in L_2(\R_+) \cap C(\R_+)$ with the norm $\| \varphi \| = \| \varphi \|_{L_2} + \| \varphi\|_{C}$. In this space introduce the operators 
\[
\aligned \relax[ \mathcal F\varphi ](y) &=\int_{0}^{\infty} \hat F_d\Big(4c^2t +\frac{m_0-\frac 3 2 -\varepsilon}{c}\log t+ \xi+y+s, t\Big) \varphi(s) d s, \\
[\mathcal G \varphi](y) &=\int_0^{\infty} G\Big(4c^2t +\frac{m_0-\frac 3 2 -\varepsilon}{c}\log t + \xi+y+s, t\Big) \varphi(s) ds,
\endaligned 
\]
where  $G(x,t)=\hat F_T(x, t) +\hat F_R(x, t)$.
The Marchenko equation  \eqref{nuzh} can be represented as  
\begin{equation} \label{eq4.3}
\varphi+\F \varphi+\mathcal G \varphi=\omega+g,
\end{equation}
where 
\begin{align*}
\omega&=\omega(y)=-\hat F_d\Big(4c^2 t +\frac{m_0-\frac 3 2 -\varepsilon}{c}\log t+\xi+y, t\Big),\\
 g&=g(y)=-\mathcal G\Big(4c^2 t+\frac{m_0-\frac 3 2 -\varepsilon}{c}\log t+\xi+y, t\Big), \\
\varphi &=\varphi(y)=B\Big(4 c^{2} t+\frac{m_0-\frac 3 2 -\varepsilon}{c}\log t+\xi, y, t\Big).
\end{align*}
The operators $\F$, $\mathcal G$ and the functions $\omega$, $g$, $\varphi$ depend on $\xi$ and $t$ as parameters. By \eqref{mainest} and \eqref{mainest1} we have 
\begin{equation} \label{eq4.5}
\|G\|+\left\|G_{\xi}\right\| \leq C t^{-(m_0 - 3/2-\varepsilon)}, \quad \|g\|+\left\|g_{\xi}\right\| \leq  C t^{-(m_0 - 3/2-\varepsilon)}, 
\end{equation}
 where $G_{\xi} = \frac{dG}{d \xi}$, $g_{\xi} = \frac{dg}{d \xi}$, and the norm is taken in 
$L_2(0, \infty) \cap C[0, \infty)$. 

Lemma \ref{zak} has the following corollary (cf.\ \cite{Kh}).
\begin{corollary}\label{cor}
The operator $\id + \F$ in the spaces $L_2(\R_+)$ and $ L_2(\R_+) \cap C[\R_+)$ is invertible,  and the norms of the operators $\RR = (\id + \F)^{-1}$ and $\RR_{\xi} = \frac{d\RR}{d \xi}$ are bounded uniformly with respect to $t$ and $\xi \geq 0$.
\end{corollary}

From this corollary it follows that \eqref{eq4.3} can be rewritten as
\begin{equation*}
\varphi+\RR \mathcal G \varphi=\RR \omega+\RR g,
\end{equation*}
where the operator $\RR \mathcal G$ is small in the norm of $ L_2(\R_+) \cap C(\R_+)$ for large $t$.
Applying successive approximation we get
\[
\varphi=\RR  \omega+\RR  g-\sum_{n=1}^{\infty}(\RR \mathcal G)^{n}[\RR  \omega+\RR  g].
\]
Differentiating  by $\xi$ yields
\begin{equation}\label{eq4.9}
\begin{aligned}
\frac{d}{d \xi} \varphi&=\frac{d}{d \xi}(\RR \omega)+\RR _{\xi} g +\RR  g_{\xi}\\
&\quad -\sum_{n=1}^{\infty} \sum_{i=0}^{n-1}(\RR  G)^{i}\left(\RR _{\xi} G+\RR  G_{\xi}\right)
(\RR  G)^{n-1-i}[\RR  \omega+\RR  g]
\\
&\quad -\sum_{n=1}^{\infty}(\RR  G)^{n}\left[\frac{d}{d \xi}(\RR  \omega)+\RR _{\xi} g+\RR  g_{\xi}\right].
\end{aligned}
\end{equation}
Taking into account estimates (\ref{eq4.5}) and Corollary \ref{cor} we obtain that the series on the r.h.s. of \eqref{eq4.9} converge in $L_2(\R_+) \cap C(\R_+)$ uniformly with respect to  $t$.
From (\ref{eq4.5}) and \eqref{eq4.9} it also follows that
\[
\frac{d}{d \xi} \varphi \Big|_{y=0}=\frac{d}{d \xi}\left(\RR \omega\right)\Big|_{y=0}+O\big(t^{-(m_0-3/2-\varepsilon)}\big).
\]
Therefore, the main contribution in the asymptotics of the solution $q(x,t)$ of the initial value problem \eqref{kdv}, \eqref{inithor} is given by the solution of \eqref{nuzh}. Together with  \eqref{solit} this finishes the  proof of Theorem \ref{maintheor1}.

\end{document}